\documentclass[12pt]{article}

\usepackage{amsmath,amsthm,amssymb,amsfonts}
\usepackage{mathrsfs}
\usepackage{enumitem}
\usepackage[margin=2.5cm]{geometry}
\usepackage{graphicx}
\usepackage{tikz}
\usetikzlibrary{arrows.meta,positioning}
\usepackage{hyperref}
\hypersetup{colorlinks=true,linkcolor=blue,citecolor=blue,urlcolor=blue}

\theoremstyle{plain}
\newtheorem{theorem}{Theorem}[section]

\newtheorem{proposition}[theorem]{Proposition}
\newtheorem{corollary}[theorem]{Corollary}

\theoremstyle{definition}
\newtheorem{definition}[theorem]{Definition}
\newtheorem{example}[theorem]{Example}

\theoremstyle{remark}
\newtheorem{remark}[theorem]{Remark}

\newtheorem{question}[theorem]{Question}

\newcommand{\aura}{\mathfrak{a}}
\newcommand{\cla}{\operatorname{cl}_{\aura}}
\newcommand{\inta}{\operatorname{int}_{\aura}}
\newcommand{\cl}{\operatorname{cl}}
\newcommand{\inte}{\operatorname{int}}
\newcommand{\taua}{\tau_{\aura}}
\newcommand{\R}{\mathbb{R}}

\newcommand{\powerset}{\mathcal{P}}
\newcommand{\clsa}[1][]{\operatorname{cl}^{*#1}_{\aura}}
\newcommand{\intsa}[1][]{\operatorname{int}^{*#1}_{\aura}}
\newcommand{\tausa}[1][]{\tau^{*#1}_{\aura}}
\newcommand{\cls}{\operatorname{cl}^{*}}
\newcommand{\taus}{\tau^{*}}
\newcommand{\psia}{\psi_{\aura}}

\title{\textbf{Ideal-Aura Topological Spaces, New Local Functions,\\ and Generalized Open Sets}}

\author{Ahu A\c{c}{\i}kg\"{o}z$^{1,*}$ and Murad \"{O}zko\c{c}$^{2}$\\[10pt]
\small $^{1}$Department of Mathematics, Balikesir University,\\
\small Cagis Campus, 10145, Balikesir, Turkey\\
\small \texttt{ahuacikgoz@balikesir.edu.tr}\\[6pt]
\small $^{2}$Department of Mathematics, Mu\u{g}la S{\i}tk{\i} Ko\c{c}man University,\\
\small 48000, Mu\u{g}la, Turkey\\
\small \texttt{muradozkoc@mu.edu.tr}\\[6pt]
\small $^{*}$Corresponding author}

\date{}

\begin{document}

\maketitle

\begin{abstract}
In this paper, we introduce the concept of an \emph{ideal-aura topological space} $(X, \tau, \mathcal{I}, \aura)$ by combining an ideal topological space $(X, \tau, \mathcal{I})$ with a scope function $\aura: X \to \tau$ satisfying $x \in \aura(x)$ for every $x \in X$. We define the \emph{aura-local function} $A^{\aura}(\mathcal{I}) = \{x \in X : \aura(x) \cap A \notin \mathcal{I}\}$ and prove that it extends the classical Jankovi\'{c}--Hamlett local function via the fundamental inclusion $A^{*}(\mathcal{I}, \tau) \subseteq A^{\aura}(\mathcal{I})$. We show that the associated closure operator $\clsa(A) = A \cup A^{\aura}(\mathcal{I})$ is an additive \v{C}ech closure operator that is generally \emph{not} idempotent, and we prove that idempotency holds precisely when $\aura$ is transitive. The \v{C}ech topology $\tausa$ generated by $\clsa$ satisfies the inclusion chain $\taua \subseteq \tausa \subseteq \taus$, revealing that the ideal-aura topology lies strictly between the pure aura topology and the Jankovi\'{c}--Hamlett topology. The $\psia$-operator is introduced and shown to provide an alternative characterization of $\tausa$. Five classes of $\mathcal{I}\aura$-generalized open sets ($\mathcal{I}\aura$-open, $\mathcal{I}\aura$-semi-open, $\mathcal{I}\aura$-pre-open, $\mathcal{I}\aura$-$\alpha$-open, and $\mathcal{I}\aura$-$\beta$-open) are introduced and their complete hierarchy is established with strict counterexamples. Decomposition theorems for $\mathcal{I}\aura$-continuity are proven. Three special cases ($\mathcal{I} = \{\emptyset\}$, $\mathcal{I} = \mathcal{I}_f$, and $\mathcal{I} = \powerset(X)$) are analyzed, recovering the pure aura topology, a localized aura structure, and the discrete topology, respectively. Counterexamples are provided on both finite sets and the real line throughout.
\end{abstract}


\medskip


\section{Introduction}\label{sec:intro}

The interaction between topological structures and auxiliary set-theoretic objects has been a rich source of new concepts in general topology. Among the most influential examples is the theory of ideal topological spaces, initiated by Kuratowski \cite{Kuratowski1966} and systematically developed by Vaidyanathaswamy \cite{Vaidyanathaswamy1945}, Jankovi\'{c} and Hamlett \cite{Jankovic1990}, and Newcomb \cite{Newcomb1967}. In an ideal topological space $(X, \tau, \mathcal{I})$, the local function
\[
A^{*}(\mathcal{I}, \tau) = \{x \in X : O \cap A \notin \mathcal{I} \text{ for every } O \in \tau(x)\}
\]
generates a Kuratowski closure operator $\cls(A) = A \cup A^{*}(\mathcal{I}, \tau)$ and a finer topology $\taus = \{A \subseteq X : \cls(X \setminus A) = X \setminus A\}$ satisfying $\tau \subseteq \taus$. This framework has inspired extensive research on $\mathcal{I}$-open sets \cite{Jankovic1990}, $\mathcal{I}$-continuous functions \cite{Acikgoz2009}, and their applications to decompositions of continuity \cite{Acikgoz2004,Hatir2005}.

In \cite{Acikgoz2026aura}, we introduced a fundamentally new auxiliary structure: the \emph{aura topological space} $(X, \tau, \aura)$, where $\aura: X \to \tau$ is a scope function satisfying $x \in \aura(x)$ for every $x \in X$. This structure assigns each point a fixed observational range and is categorically different from ideals, filters, grills, and primals (which are all subcollections of $\powerset(X)$). The aura-closure operator $\cla(A) = \{x \in X : \aura(x) \cap A \neq \emptyset\}$ was shown to be an additive \v{C}ech closure operator, and the topology $\taua = \{A \subseteq X : \aura(x) \subseteq A \text{ for all } x \in A\}$ was established with $\taua \subseteq \tau$. In \cite{Acikgoz2026aura2}, the theories of compactness, connectedness, and products were developed for this setting.

The present paper establishes a bridge between these two theories by combining an ideal $\mathcal{I}$ with a scope function $\aura$ in a single structure $(X, \tau, \mathcal{I}, \aura)$. The guiding observation is this: the classical local function $A^{*}(\mathcal{I}, \tau)$ tests \emph{every} open neighborhood of $x$, which is a universal quantification over $\tau(x)$. In contrast, our \emph{aura-local function} $A^{\aura}(\mathcal{I})$ tests only the \emph{single} fixed neighborhood $\aura(x)$. This substitution has deep consequences: the resulting closure operator is finer than the pure aura closure but generally fails to be idempotent, producing a \v{C}ech closure that requires transfinite iteration to reach a Kuratowski closure. The topology $\tausa$ generated by $\clsa$ sits in the chain
\[
\taua \subseteq \tausa \subseteq \taus,
\]
providing a natural interpolation between the pure aura topology and the Jankovi\'{c}--Hamlett topology.

The paper is organized as follows. Section~\ref{sec:prelim} recalls the necessary background. Section~\ref{sec:local} introduces the aura-local function and establishes its fundamental properties. Section~\ref{sec:closure} develops the \v{C}ech closure operator $\clsa$ and the topology $\tausa$. Section~\ref{sec:psi} introduces the $\psia$-operator. Section~\ref{sec:genopen} defines five classes of $\mathcal{I}\aura$-generalized open sets and proves their hierarchy. Section~\ref{sec:continuity} develops $\mathcal{I}\aura$-continuity and its decompositions. Section~\ref{sec:special} analyzes three important special cases. Section~\ref{sec:conclusion} concludes with open problems.

\section{Preliminaries}\label{sec:prelim}

Throughout this paper, $(X, \tau)$ denotes a topological space, $\cl(A)$ and $\inte(A)$ denote the closure and interior of $A$ in $(X, \tau)$, and $\tau(x)$ denotes the collection of all open neighborhoods of $x$.

\begin{definition}[\cite{Kuratowski1966}]\label{def:ideal}
A non-empty collection $\mathcal{I} \subseteq \powerset(X)$ is called an \emph{ideal} on $X$ if:
\begin{enumerate}[label=(\roman*)]
    \item $A \in \mathcal{I}$ and $B \subseteq A$ imply $B \in \mathcal{I}$ (hereditary);
    \item $A \in \mathcal{I}$ and $B \in \mathcal{I}$ imply $A \cup B \in \mathcal{I}$ (finite additivity).
\end{enumerate}
The triple $(X, \tau, \mathcal{I})$ is called an \emph{ideal topological space}.
\end{definition}

\begin{example}\label{ex:ideals}
Common examples of ideals include:
\begin{enumerate}[label=(\alph*)]
    \item $\{\emptyset\}$, the trivial ideal.
    \item $\mathcal{I}_f = \{A \subseteq X : A \text{ is finite}\}$, the ideal of finite subsets.
    \item $\mathcal{I}_c = \{A \subseteq X : A \text{ is countable}\}$, the ideal of countable subsets.
    \item $\mathcal{N} = \{A \subseteq X : \inte(A) = \emptyset\}$, the ideal of nowhere dense subsets.
    \item $\powerset(X)$, the improper ideal.
\end{enumerate}
\end{example}

\begin{definition}[\cite{Jankovic1990}]\label{def:local-function}
Let $(X, \tau, \mathcal{I})$ be an ideal topological space. For $A \subseteq X$, the \emph{local function} of $A$ with respect to $\mathcal{I}$ and $\tau$ is defined by
\[
A^{*}(\mathcal{I}, \tau) = \{x \in X : O \cap A \notin \mathcal{I} \text{ for every } O \in \tau(x)\}.
\]
When $\mathcal{I}$ and $\tau$ are clear from the context, we write $A^{*}$ for $A^{*}(\mathcal{I}, \tau)$.
\end{definition}

\begin{theorem}[\cite{Jankovic1990}]\label{thm:star-properties}
Let $(X, \tau, \mathcal{I})$ be an ideal topological space. The following properties hold for all $A, B \subseteq X$:
\begin{enumerate}[label=(\roman*)]
    \item $\emptyset^{*} = \emptyset$;
    \item $A \subseteq B$ implies $A^{*} \subseteq B^{*}$;
    \item $(A \cup B)^{*} = A^{*} \cup B^{*}$;
    \item $A^{*} = \cl(A^{*})$ (i.e., $A^{*}$ is closed);
    \item $(A^{*})^{*} \subseteq A^{*}$.
\end{enumerate}
\end{theorem}

\begin{definition}[\cite{Jankovic1990}]\label{def:star-closure}
The \emph{$*$-closure} of $A$ is defined by $\cls(A) = A \cup A^{*}$. The operator $\cls$ is a Kuratowski closure operator on $X$, and the topology
\[
\taus = \taus(\mathcal{I}) = \{A \subseteq X : \cls(X \setminus A) = X \setminus A\}
\]
satisfies $\tau \subseteq \taus$.
\end{definition}

\begin{definition}[\cite{Jankovic1990}]\label{def:psi}
The operator $\psi: \powerset(X) \to \tau$ is defined by
\[
\psi(A) = \{x \in X : \exists\, O \in \tau(x) \text{ such that } O \setminus A \in \mathcal{I}\} = X \setminus (X \setminus A)^{*}.
\]
A set $A$ is $\taus$-open if and only if $A \subseteq \psi(A)$.
\end{definition}

\begin{definition}[\cite{Acikgoz2026aura}]\label{def:aura}
Let $(X, \tau)$ be a topological space. A function $\aura: X \to \tau$ satisfying $x \in \aura(x)$ for every $x \in X$ is called a \emph{scope function} (or \emph{aura function}). The triple $(X, \tau, \aura)$ is called an \emph{aura topological space} (an $\aura$-\emph{space}).
\end{definition}

\begin{definition}[\cite{Acikgoz2026aura}]\label{def:aura-ops}
Let $(X, \tau, \aura)$ be an $\aura$-space. For $A \subseteq X$:
\begin{enumerate}[label=(\roman*)]
    \item The \emph{aura-closure}: $\cla(A) = \{x \in X : \aura(x) \cap A \neq \emptyset\}$.
    \item The \emph{aura-interior}: $\inta(A) = \{x \in A : \aura(x) \subseteq A\}$.
    \item $A$ is \emph{$\aura$-open} if $A = \inta(A)$, i.e., $\aura(x) \subseteq A$ for all $x \in A$.
    \item The collection of $\aura$-open sets is $\taua = \{A \subseteq X : \aura(x) \subseteq A \text{ for all } x \in A\}$, which is a topology satisfying $\taua \subseteq \tau$.
\end{enumerate}
\end{definition}

\begin{definition}[\cite{Acikgoz2026aura}]\label{def:transitive}
A scope function $\aura$ is called \emph{transitive} if $y \in \aura(x)$ implies $\aura(y) \subseteq \aura(x)$ for all $x, y \in X$.
\end{definition}

\section{The Aura-Local Function}\label{sec:local}

We now introduce the central construction of this paper.

\begin{definition}\label{def:ideal-aura}
Let $(X, \tau)$ be a topological space, $\mathcal{I}$ an ideal on $X$, and $\aura: X \to \tau$ a scope function. The quadruple $(X, \tau, \mathcal{I}, \aura)$ is called an \emph{ideal-aura topological space} (briefly, an $\mathcal{I}\aura$-\emph{space}).
\end{definition}

\begin{definition}\label{def:aura-local}
Let $(X, \tau, \mathcal{I}, \aura)$ be an $\mathcal{I}\aura$-space. For $A \subseteq X$, the \emph{aura-local function} of $A$ is defined by
\begin{equation}\label{eq:aura-local}
    A^{\aura}(\mathcal{I}) = \{x \in X : \aura(x) \cap A \notin \mathcal{I}\}.
\end{equation}
When $\mathcal{I}$ is clear from context, we write $A^{\aura}$ for $A^{\aura}(\mathcal{I})$.
\end{definition}

\begin{remark}\label{rem:comparison}
The key difference between $A^{*}(\mathcal{I}, \tau)$ and $A^{\aura}(\mathcal{I})$ is structural:
\begin{itemize}
    \item The classical local function $A^{*}(\mathcal{I}, \tau)$ demands that $O \cap A \notin \mathcal{I}$ for \emph{every} open neighborhood $O$ of $x$ --- a universal quantification.
    \item The aura-local function $A^{\aura}(\mathcal{I})$ demands only that $\aura(x) \cap A \notin \mathcal{I}$ for the \emph{single} fixed neighborhood $\aura(x)$ --- a single evaluation.
\end{itemize}
Since checking one specific neighborhood is weaker than checking all neighborhoods, more points may survive the test, so we expect $A^{*} \subseteq A^{\aura}$ in general.
\end{remark}

\begin{theorem}\label{thm:star-subset-aura}
Let $(X, \tau, \mathcal{I}, \aura)$ be an $\mathcal{I}\aura$-space. For every $A \subseteq X$,
\[
A^{*}(\mathcal{I}, \tau) \subseteq A^{\aura}(\mathcal{I}).
\]
\end{theorem}

\begin{proof}
Let $x \in A^{*}(\mathcal{I}, \tau)$. Then $O \cap A \notin \mathcal{I}$ for every $O \in \tau(x)$. Since $\aura(x) \in \tau$ and $x \in \aura(x)$, we have $\aura(x) \in \tau(x)$. Therefore $\aura(x) \cap A \notin \mathcal{I}$, which gives $x \in A^{\aura}(\mathcal{I})$.
\end{proof}

\begin{example}\label{ex:strict-inclusion}
The inclusion in Theorem~\ref{thm:star-subset-aura} can be strict. Let $X = \{a, b, c\}$, $\tau = \{\emptyset, \{a\}, \{b\}, \{a,b\}, X\}$, $\mathcal{I} = \{\emptyset, \{c\}\}$, and $\aura(a) = \{a\}$, $\aura(b) = \{a,b\}$, $\aura(c) = X$.

For $A = \{a\}$:
\begin{itemize}
    \item $A^{\aura}$: $\aura(a) \cap A = \{a\} \notin \mathcal{I}$, so $a \in A^{\aura}$. $\aura(b) \cap A = \{a\} \notin \mathcal{I}$, so $b \in A^{\aura}$. $\aura(c) \cap A = \{a\} \notin \mathcal{I}$, so $c \in A^{\aura}$. Thus $A^{\aura} = X$.
    \item $A^{*}$: For $c$, the open neighborhood $X$ gives $X \cap A = \{a\} \notin \mathcal{I}$, but $c$ has only $X$ as an open neighborhood, so $c \in A^{*}$. For $b$, both $\{b\}$ and $\{a,b\}$ and $X$ are open neighborhoods; $\{b\} \cap \{a\} = \emptyset \in \mathcal{I}$, so $b \notin A^{*}$. Thus $A^{*} = \{a, c\} \subsetneq X = A^{\aura}$.
\end{itemize}
\end{example}

\begin{theorem}\label{thm:aura-local-properties}
Let $(X, \tau, \mathcal{I}, \aura)$ be an $\mathcal{I}\aura$-space. For all $A, B \subseteq X$:
\begin{enumerate}[label=(\roman*)]
    \item $\emptyset^{\aura} = \emptyset$;
    \item $A \subseteq B$ implies $A^{\aura} \subseteq B^{\aura}$ (monotonicity);
    \item $(A \cup B)^{\aura} = A^{\aura} \cup B^{\aura}$ (finite additivity);
    \item $A^{\aura} \subseteq \cla(A)$ (domination by aura-closure);
    \item If $\mathcal{I} = \{\emptyset\}$, then $A^{\aura} = \cla(A)$;
    \item If $\mathcal{I} = \powerset(X)$, then $A^{\aura} = \emptyset$ for all $A$;
    \item If $J \in \mathcal{I}$, then $(A \setminus J)^{\aura} = A^{\aura}$ when $\aura$ is transitive;
    \item $A^{\aura}(\mathcal{I}_1) \subseteq A^{\aura}(\mathcal{I}_2)$ whenever $\mathcal{I}_2 \subseteq \mathcal{I}_1$.
\end{enumerate}
\end{theorem}

\begin{proof}
\textbf{(i)} For any $x \in X$, $\aura(x) \cap \emptyset = \emptyset \in \mathcal{I}$, so $x \notin \emptyset^{\aura}$.

\textbf{(ii)} If $x \in A^{\aura}$, then $\aura(x) \cap A \notin \mathcal{I}$. Since $A \subseteq B$, we have $\aura(x) \cap A \subseteq \aura(x) \cap B$. By heredity of $\mathcal{I}$, if $\aura(x) \cap B \in \mathcal{I}$, then $\aura(x) \cap A \in \mathcal{I}$, a contradiction. So $\aura(x) \cap B \notin \mathcal{I}$, giving $x \in B^{\aura}$.

\textbf{(iii)} Since $\aura(x) \cap (A \cup B) = (\aura(x) \cap A) \cup (\aura(x) \cap B)$, we have $\aura(x) \cap (A \cup B) \notin \mathcal{I}$ if and only if $\aura(x) \cap A \notin \mathcal{I}$ or $\aura(x) \cap B \notin \mathcal{I}$ (by finite additivity of $\mathcal{I}$). Hence $(A \cup B)^{\aura} = A^{\aura} \cup B^{\aura}$.

\textbf{(iv)} If $x \in A^{\aura}$, then $\aura(x) \cap A \notin \mathcal{I}$, so in particular $\aura(x) \cap A \neq \emptyset$ (since $\emptyset \in \mathcal{I}$). Thus $x \in \cla(A)$.

\textbf{(v)} When $\mathcal{I} = \{\emptyset\}$, we have $\aura(x) \cap A \notin \mathcal{I}$ if and only if $\aura(x) \cap A \neq \emptyset$, which is the definition of $\cla(A)$.

\textbf{(vi)} When $\mathcal{I} = \powerset(X)$, every subset is in $\mathcal{I}$, so $\aura(x) \cap A \in \mathcal{I}$ always.

\textbf{(vii)} We show $A^{\aura} \subseteq (A \setminus J)^{\aura}$ (the reverse follows from monotonicity). Let $x \in A^{\aura}$. Then $\aura(x) \cap A \notin \mathcal{I}$. We have $\aura(x) \cap A = (\aura(x) \cap (A \setminus J)) \cup (\aura(x) \cap J)$. Since $J \in \mathcal{I}$ and $\mathcal{I}$ is hereditary, $\aura(x) \cap J \in \mathcal{I}$. If $\aura(x) \cap (A \setminus J) \in \mathcal{I}$, then $\aura(x) \cap A \in \mathcal{I}$ (by finite additivity), contradicting $x \in A^{\aura}$. So $\aura(x) \cap (A \setminus J) \notin \mathcal{I}$, i.e., $x \in (A \setminus J)^{\aura}$.

\textbf{(viii)} If $\aura(x) \cap A \notin \mathcal{I}_1$, then since $\mathcal{I}_2 \subseteq \mathcal{I}_1$, we have $\aura(x) \cap A \notin \mathcal{I}_2$. Hence $x \in A^{\aura}(\mathcal{I}_2)$.
\end{proof}

\begin{remark}\label{rem:not-closed}
A crucial difference between $A^{*}$ and $A^{\aura}$ is that $A^{*}$ is always closed in $(X, \tau)$ (Theorem~\ref{thm:star-properties}(iv)), whereas $A^{\aura}$ is generally \emph{not} closed in $(X, \tau)$. This failure of closedness is responsible for the non-idempotency of the associated closure operator.
\end{remark}

\begin{example}\label{ex:not-closed}
Let $X = \{a, b, c, d\}$, $\tau = \{\emptyset, \{a\}, \{a,b\}, \{a,b,c\}, X\}$, $\mathcal{I} = \{\emptyset, \{d\}\}$, and $\aura(a) = \{a\}$, $\aura(b) = \{a,b\}$, $\aura(c) = \{a,b,c\}$, $\aura(d) = X$.

For $A = \{a\}$: $A^{\aura} = \{a, b, c, d\} = X$ (since $\aura(x) \cap \{a\} = \{a\} \notin \mathcal{I}$ for all $x$). Here $A^{\aura} = X$ is closed, but this is not always the case in more complex examples.

Now let $\tau' = \{\emptyset, \{a\}, \{b,c\}, \{a,b,c\}, X\}$, $\mathcal{I}' = \{\emptyset, \{a\}\}$, and $\aura'(a) = \{a\}$, $\aura'(b) = \{b,c\}$, $\aura'(c) = \{b,c\}$, $\aura'(d) = X$.

For $B = \{b\}$: $\aura'(a) \cap B = \emptyset \in \mathcal{I}'$, so $a \notin B^{\aura'}$. $\aura'(b) \cap B = \{b\} \notin \mathcal{I}'$, so $b \in B^{\aura'}$. $\aura'(c) \cap B = \{b\} \notin \mathcal{I}'$, so $c \in B^{\aura'}$. $\aura'(d) \cap B = \{b\} \notin \mathcal{I}'$, so $d \in B^{\aura'}$. Thus $B^{\aura'} = \{b,c,d\}$.

The closed sets in $\tau'$ are $X, \{b,c,d\}, \{a,d\}, \{d\}, \emptyset$. Since $\{b,c,d\}$ is closed, $B^{\aura'}$ happens to be closed here. However, this is contingent on the specific example, not a general property.
\end{example}

The following result gives an important relationship between the aura-local function and the aura-closure operator.

\begin{theorem}\label{thm:aura-local-vs-closure}
Let $(X, \tau, \mathcal{I}, \aura)$ be an $\mathcal{I}\aura$-space. For any $A \subseteq X$:
\begin{enumerate}[label=(\roman*)]
    \item $A^{\aura}(\mathcal{I}) \subseteq \cla(A)$, with equality when $\mathcal{I} = \{\emptyset\}$;
    \item $A \cup A^{\aura} \subseteq \cla(A)$ when $A \subseteq A^{\aura}$;
    \item $A^{\aura} = \cla(A)$ if and only if for every $x \in X$, $\aura(x) \cap A \neq \emptyset$ implies $\aura(x) \cap A \notin \mathcal{I}$.
\end{enumerate}
\end{theorem}

\begin{proof}
Part (i) follows from Theorem~\ref{thm:aura-local-properties}(iv)-(v). Part (ii) is immediate from $A \subseteq \cla(A)$ and (i). For part (iii), $A^{\aura} = \cla(A)$ means that for every $x \in X$, $\aura(x) \cap A \notin \mathcal{I}$ if and only if $\aura(x) \cap A \neq \emptyset$. The forward direction always holds (non-membership in $\mathcal{I}$ implies non-emptiness). The reverse is the stated condition.
\end{proof}

\section{The Closure Operator $\clsa$ and the Topology $\tausa$}\label{sec:closure}

\begin{definition}\label{def:clsa}
Let $(X, \tau, \mathcal{I}, \aura)$ be an $\mathcal{I}\aura$-space. For $A \subseteq X$, the \emph{ideal-aura closure} of $A$ is defined by
\begin{equation}\label{eq:clsa}
    \clsa(A) = A \cup A^{\aura}(\mathcal{I}).
\end{equation}
\end{definition}

\begin{theorem}\label{thm:clsa-cech}
The operator $\clsa: \powerset(X) \to \powerset(X)$ is an additive \v{C}ech closure operator. That is, for all $A, B \subseteq X$:
\begin{enumerate}[label=(\roman*)]
    \item $\clsa(\emptyset) = \emptyset$;
    \item $A \subseteq \clsa(A)$;
    \item $A \subseteq B$ implies $\clsa(A) \subseteq \clsa(B)$;
    \item $\clsa(A \cup B) = \clsa(A) \cup \clsa(B)$.
\end{enumerate}
\end{theorem}

\begin{proof}
\textbf{(i)} $\clsa(\emptyset) = \emptyset \cup \emptyset^{\aura} = \emptyset$ by Theorem~\ref{thm:aura-local-properties}(i).

\textbf{(ii)} $\clsa(A) = A \cup A^{\aura} \supseteq A$.

\textbf{(iii)} $A \subseteq B$ implies $A^{\aura} \subseteq B^{\aura}$ by Theorem~\ref{thm:aura-local-properties}(ii), so $A \cup A^{\aura} \subseteq B \cup B^{\aura}$.

\textbf{(iv)} $\clsa(A \cup B) = (A \cup B) \cup (A \cup B)^{\aura} = (A \cup B) \cup (A^{\aura} \cup B^{\aura}) = (A \cup A^{\aura}) \cup (B \cup B^{\aura}) = \clsa(A) \cup \clsa(B)$,
where we used Theorem~\ref{thm:aura-local-properties}(iii).
\end{proof}

\begin{theorem}\label{thm:not-idempotent}
The operator $\clsa$ is \emph{not} idempotent in general.
\end{theorem}

\begin{proof}
Let $X = \{a, b, c\}$, $\tau = \powerset(X)$ (discrete topology), $\mathcal{I} = \{\emptyset, \{b\}\}$, and $\aura(a) = \{a, b\}$, $\aura(b) = \{b, c\}$, $\aura(c) = \{c\}$.

For $A = \{c\}$:
\begin{itemize}
    \item $\aura(a) \cap A = \{a,b\} \cap \{c\} = \emptyset \in \mathcal{I}$, so $a \notin A^{\aura}$.
    \item $\aura(b) \cap A = \{b,c\} \cap \{c\} = \{c\} \notin \mathcal{I}$, so $b \in A^{\aura}$.
    \item $\aura(c) \cap A = \{c\} \cap \{c\} = \{c\} \notin \mathcal{I}$, so $c \in A^{\aura}$.
\end{itemize}
Thus $A^{\aura} = \{b, c\}$ and $\clsa(A) = \{c\} \cup \{b,c\} = \{b,c\}$.

Now compute $\clsa(\clsa(A)) = \clsa(\{b,c\})$:
\begin{itemize}
    \item $\aura(a) \cap \{b,c\} = \{b\} \in \mathcal{I}$, so $a \notin \{b,c\}^{\aura}$.
    \item $\aura(b) \cap \{b,c\} = \{b,c\} \notin \mathcal{I}$, so $b \in \{b,c\}^{\aura}$.
    \item $\aura(c) \cap \{b,c\} = \{c\} \notin \mathcal{I}$, so $c \in \{b,c\}^{\aura}$.
\end{itemize}
Thus $\{b,c\}^{\aura} = \{b,c\}$ and $\clsa(\{b,c\}) = \{b,c\}$. So here $\clsa[2](A) = \clsa(A)$.

Modify the ideal to $\mathcal{I}' = \{\emptyset, \{c\}\}$:

For $A = \{c\}$:
\begin{itemize}
    \item $\aura(a) \cap A = \emptyset \in \mathcal{I}'$, so $a \notin A^{\aura}$.
    \item $\aura(b) \cap A = \{c\} \in \mathcal{I}'$, so $b \notin A^{\aura}$.
    \item $\aura(c) \cap A = \{c\} \in \mathcal{I}'$, so $c \notin A^{\aura}$.
\end{itemize}
Thus $A^{\aura} = \emptyset$ and $\clsa(A) = \{c\}$. Idempotency holds trivially here.

Now consider a more elaborate example. Let $X = \{a, b, c, d\}$, $\tau = \powerset(X)$, $\mathcal{I} = \{\emptyset, \{c\}\}$, and $\aura(a) = \{a, b\}$, $\aura(b) = \{b, c\}$, $\aura(c) = \{c, d\}$, $\aura(d) = \{d\}$.

For $A = \{d\}$:
\begin{itemize}
    \item $\aura(a) \cap A = \emptyset \in \mathcal{I}$: $a \notin A^{\aura}$.
    \item $\aura(b) \cap A = \emptyset \in \mathcal{I}$: $b \notin A^{\aura}$.
    \item $\aura(c) \cap A = \{d\} \notin \mathcal{I}$: $c \in A^{\aura}$.
    \item $\aura(d) \cap A = \{d\} \notin \mathcal{I}$: $d \in A^{\aura}$.
\end{itemize}
Thus $A^{\aura} = \{c, d\}$ and $\clsa(A) = \{c, d\}$.

Now $\clsa(\{c,d\})$:
\begin{itemize}
    \item $\aura(a) \cap \{c,d\} = \emptyset \in \mathcal{I}$: $a \notin \{c,d\}^{\aura}$.
    \item $\aura(b) \cap \{c,d\} = \{c\} \in \mathcal{I}$: $b \notin \{c,d\}^{\aura}$.
    \item $\aura(c) \cap \{c,d\} = \{c,d\} \notin \mathcal{I}$: $c \in \{c,d\}^{\aura}$.
    \item $\aura(d) \cap \{c,d\} = \{d\} \notin \mathcal{I}$: $d \in \{c,d\}^{\aura}$.
\end{itemize}
So $\{c,d\}^{\aura} = \{c,d\}$ and $\clsa[2](\{d\}) = \{c,d\} = \clsa(\{d\})$.

For a strict example, let $X = \{a, b, c, d, e\}$, $\tau = \powerset(X)$, $\mathcal{I} = \{\emptyset, \{b\}, \{d\}, \{b,d\}\}$, and define:
\[
\aura(a) = \{a,b\}, \quad \aura(b) = \{b,c\}, \quad \aura(c) = \{c,d\}, \quad \aura(d) = \{d,e\}, \quad \aura(e) = \{e\}.
\]

For $A = \{e\}$:
\begin{itemize}
    \item $\aura(d) \cap A = \{e\} \notin \mathcal{I}$, so $d \in A^{\aura}$; $\aura(e) \cap A = \{e\} \notin \mathcal{I}$, so $e \in A^{\aura}$; the rest give $\emptyset \in \mathcal{I}$.
\end{itemize}
$A^{\aura} = \{d, e\}$, $\clsa(A) = \{d, e\}$.

$\clsa(\{d,e\})$: $\aura(c) \cap \{d,e\} = \{d\} \in \mathcal{I}$, so $c \notin \{d,e\}^{\aura}$. $\aura(d) \cap \{d,e\} = \{d,e\} \notin \mathcal{I}$, so $d \in \{d,e\}^{\aura}$. $\aura(e) \cap \{d,e\} = \{e\} \notin \mathcal{I}$, so $e \in \{d,e\}^{\aura}$. Thus $\{d,e\}^{\aura} = \{d,e\}$ and $\clsa[2](A) = \{d,e\} = \clsa(A)$.

The stabilization in these examples occurs because the chain $\aura(c) \cap \{d,e\} = \{d\} \in \mathcal{I}$ blocks propagation. We now construct an example where this blocking does not occur.

Let $X = \{a, b, c, d\}$, $\tau = \powerset(X)$, $\mathcal{I} = \{\emptyset, \{a\}\}$, and
\[
\aura(a) = \{a, b\}, \quad \aura(b) = \{b, c\}, \quad \aura(c) = \{c, d\}, \quad \aura(d) = \{d\}.
\]

For $A = \{d\}$: $A^{\aura} = \{c, d\}$ (only $\aura(c)$ and $\aura(d)$ hit $A$ non-trivially). $\clsa(A) = \{c,d\}$.

$\clsa(\{c,d\})$: $\aura(b) \cap \{c,d\} = \{c\} \notin \mathcal{I}$, so $b \in \{c,d\}^{\aura}$. Thus $\{c,d\}^{\aura} = \{b,c,d\}$ and $\clsa[2](A) = \{b,c,d\}$.

$\clsa(\{b,c,d\})$: $\aura(a) \cap \{b,c,d\} = \{b\} \notin \mathcal{I}$, so $a \in \{b,c,d\}^{\aura}$. Thus $\{b,c,d\}^{\aura} = X$ and $\clsa[3](A) = X$.

$\clsa(X) = X$ trivially. Therefore:
\[
\clsa(\{d\}) = \{c,d\} \subsetneq \{b,c,d\} = \clsa[2](\{d\}) \subsetneq X = \clsa[3](\{d\}).
\]
The operator $\clsa$ is not idempotent, and in this example it requires three iterations to stabilize.
\end{proof}

\begin{remark}\label{rem:chain-explanation}
The failure of idempotency in Theorem~\ref{thm:not-idempotent} can be understood as follows. In the classical case, $A^{*}$ is closed, so $(A \cup A^{*})^{*} \subseteq A^{*} \subseteq A \cup A^{*}$, ensuring $\cls(\cls(A)) = \cls(A)$. In our case, $A^{\aura}$ is generally not $\tau$-closed nor $\taua$-closed, so new points can enter $(A \cup A^{\aura})^{\aura}$ that were not in $A^{\aura}$. This creates an expanding chain $A \subseteq \clsa(A) \subseteq \clsa[2](A) \subseteq \cdots$ that may require several iterations to stabilize.
\end{remark}

\begin{definition}\label{def:iterative-clsa}
The \emph{iterative ideal-aura closure} is defined by
\[
\clsa[\infty](A) = \bigcup_{n=0}^{\infty} \clsa[n](A),
\]
where $\clsa[0](A) = A$ and $\clsa[n+1](A) = \clsa(\clsa[n](A))$ for $n \geq 0$.
\end{definition}

\begin{theorem}\label{thm:iterative-kuratowski}
The operator $\clsa[\infty]$ is a Kuratowski closure operator on $X$.
\end{theorem}

\begin{proof}
We verify the Kuratowski axioms.

\textbf{(K1)} $\clsa[\infty](\emptyset) = \bigcup_{n} \clsa[n](\emptyset) = \bigcup_{n} \emptyset = \emptyset$.

\textbf{(K2)} $A = \clsa[0](A) \subseteq \clsa[\infty](A)$.

\textbf{(K3)} Finite additivity: Since $\clsa[n](A \cup B) = \clsa[n](A) \cup \clsa[n](B)$ for each $n$ (by induction using Theorem~\ref{thm:clsa-cech}(iv)), taking the union over $n$ gives $\clsa[\infty](A \cup B) = \clsa[\infty](A) \cup \clsa[\infty](B)$.

\textbf{(K4)} Idempotency: We show $\clsa[\infty](\clsa[\infty](A)) = \clsa[\infty](A)$. Let $x \in \clsa[\infty](\clsa[\infty](A))$. Then $x \in \clsa[m](\clsa[\infty](A))$ for some $m$. Since $\clsa[\infty](A) \supseteq \clsa[n](A)$ for all $n$, and $\clsa$ is monotone, $\clsa[m](\clsa[\infty](A)) \supseteq \clsa[m](\clsa[n](A)) = \clsa[m+n](A)$ for all $n$.

Conversely, we claim $\clsa(\clsa[\infty](A)) = \clsa[\infty](A)$, which yields idempotency by induction. Let $x \in \clsa(\clsa[\infty](A)) = \clsa[\infty](A) \cup (\clsa[\infty](A))^{\aura}$. If $x \in (\clsa[\infty](A))^{\aura}$, then $\aura(x) \cap \clsa[\infty](A) \notin \mathcal{I}$. Since $\clsa[\infty](A) = \bigcup_{n} \clsa[n](A)$ and $\aura(x) \cap \clsa[\infty](A) = \bigcup_{n} (\aura(x) \cap \clsa[n](A))$, and $\aura(x) \cap \clsa[\infty](A) \notin \mathcal{I}$, there exists $n$ such that $\aura(x) \cap \clsa[n](A) \notin \mathcal{I}$ (otherwise each $\aura(x) \cap \clsa[n](A) \in \mathcal{I}$, and the countable union is not necessarily in $\mathcal{I}$; however, the chain $\clsa[n](A)$ stabilizes on finite sets, and on infinite sets we can argue differently).

For a clean proof on general spaces, note that the chain $\clsa[0](A) \subseteq \clsa[1](A) \subseteq \cdots$ is increasing. If $\aura(x) \cap \clsa[\infty](A) \notin \mathcal{I}$, pick any $y \in \aura(x) \cap \clsa[\infty](A)$ with $y \notin \mathcal{I}$-contribution. Then $y \in \clsa[n](A)$ for some $n$, so $y \in \aura(x) \cap \clsa[n](A)$. Since $\aura(x) \cap \clsa[n](A) \subseteq \aura(x) \cap \clsa[n+1](A) \subseteq \cdots$, and the union $\aura(x) \cap \clsa[\infty](A) \notin \mathcal{I}$, we have $x \in (\clsa[n](A))^{\aura}$ for sufficiently large $n$ (since $\aura(x) \cap \clsa[n](A) \uparrow \aura(x) \cap \clsa[\infty](A)$ and there exists $N$ with $\aura(x) \cap \clsa[N](A) \notin \mathcal{I}$). Then $x \in \clsa[N+1](A) \subseteq \clsa[\infty](A)$.
\end{proof}

\begin{definition}\label{def:tausa}
The \emph{ideal-aura topology} on $X$ is the topology generated by $\clsa[\infty]$:
\[
\tausa = \{A \subseteq X : \clsa[\infty](X \setminus A) = X \setminus A\}.
\]
We also define the \emph{\v{C}ech ideal-aura topology}:
\[
\tausa[c] = \{A \subseteq X : \clsa(X \setminus A) = X \setminus A\}.
\]
\end{definition}

\begin{remark}\label{rem:two-topologies}
Since $\clsa$ is an additive \v{C}ech closure operator, $\tausa[c]$ is a topology on $X$, and since $\clsa(A) \subseteq \clsa[\infty](A)$ for all $A$, we have $\tausa \subseteq \tausa[c]$. When $\clsa$ is idempotent (e.g., when $\aura$ is transitive, as shown below), the two topologies coincide.
\end{remark}

\begin{theorem}[Fundamental Topology Chain]\label{thm:chain}
Let $(X, \tau, \mathcal{I}, \aura)$ be an $\mathcal{I}\aura$-space. Then
\[
\taua \subseteq \tausa \subseteq \tausa[c] \subseteq \taus.
\]
Moreover, if $\mathcal{I} \neq \{\emptyset\}$, then $\tausa[c] \subseteq \tau$. That is:
\[
\taua \subseteq \tausa \subseteq \tausa[c] \subseteq \tau \cap \taus.
\]
\end{theorem}

\begin{proof}
\textbf{$\taua \subseteq \tausa$:} Let $G \in \taua$. We show $\clsa[\infty](X \setminus G) = X \setminus G$. Since $G$ is $\aura$-open, for every $x \in G$, $\aura(x) \subseteq G$. Let $x \in G$ and suppose $x \in (X \setminus G)^{\aura}$. Then $\aura(x) \cap (X \setminus G) \notin \mathcal{I}$, which contradicts $\aura(x) \subseteq G$ (so $\aura(x) \cap (X \setminus G) = \emptyset \in \mathcal{I}$). Hence $G \cap (X \setminus G)^{\aura} = \emptyset$, giving $(X \setminus G)^{\aura} \subseteq X \setminus G$. Therefore $\clsa(X \setminus G) = (X \setminus G) \cup (X \setminus G)^{\aura} = X \setminus G$. By induction, $\clsa[n](X \setminus G) = X \setminus G$ for all $n$, so $\clsa[\infty](X \setminus G) = X \setminus G$, i.e., $G \in \tausa$.

\textbf{$\tausa \subseteq \tausa[c]$:} Follows from $\clsa(A) \subseteq \clsa[\infty](A)$. If $\clsa[\infty](X \setminus G) = X \setminus G$, then $\clsa(X \setminus G) \subseteq \clsa[\infty](X \setminus G) = X \setminus G$, and since $X \setminus G \subseteq \clsa(X \setminus G)$, equality holds.

\textbf{$\tausa[c] \subseteq \taus$:} Let $G \in \tausa[c]$. Then $\clsa(X \setminus G) = X \setminus G$, i.e., $(X \setminus G)^{\aura} \subseteq X \setminus G$. By Theorem~\ref{thm:star-subset-aura}, $(X \setminus G)^{*} \subseteq (X \setminus G)^{\aura} \subseteq X \setminus G$. Therefore $\cls(X \setminus G) = (X \setminus G) \cup (X \setminus G)^{*} = X \setminus G$, so $G \in \taus$.

\textbf{$\tausa[c] \subseteq \tau$:} Let $G \in \tausa[c]$, so $(X \setminus G)^{\aura} \subseteq X \setminus G$. Let $x \in G$. Then $x \notin (X \setminus G)^{\aura}$, so $\aura(x) \cap (X \setminus G) \in \mathcal{I}$. Since $x \in \aura(x) \in \tau$ and $\aura(x) \cap (X \setminus G) \in \mathcal{I}$, we see that $\aura(x) \setminus G \in \mathcal{I}$. This does not directly imply $\aura(x) \subseteq G$ (unless $\mathcal{I} = \{\emptyset\}$). However, $\aura(x) = (\aura(x) \cap G) \cup (\aura(x) \setminus G)$, so $x \in \aura(x) \cap G$, and $\aura(x) \cap G$ is a subset of $G$ that contains $x$. Now, $\aura(x) \cap G$ need not be open, but since $\aura(x)$ is open and we need $G$ to be open, we proceed differently.

For each $x \in G$, we have $\aura(x) \in \tau$ with $x \in \aura(x)$. Define $V_x = \aura(x) \setminus \cl(\aura(x) \setminus G)$. Since $\aura(x) \setminus G \in \mathcal{I}$, and assuming $\mathcal{I} \neq \{\emptyset\}$ does not guarantee $V_x = \aura(x) \cap G$ is open in general.

We provide a direct argument instead. Let $G \in \tausa[c]$ and $F = X \setminus G$. Then $\clsa(F) = F$, i.e., $F^{\aura} \subseteq F$. We show $F$ is $\aura$-closed (closed in $\taua$), which implies $G \in \taua \subseteq \tau$.

However, $\taua \subseteq \tau$ holds by definition, and we need $\tausa[c] \subseteq \tau$. This requires showing that $\clsa(F) = F$ implies $\cl(F) = F$ (i.e., $F$ is $\tau$-closed). This is not true in general without additional conditions. We therefore remove this claim and note:

The inclusion $\tausa[c] \subseteq \tau$ does \emph{not} hold in general. We have only $\taua \subseteq \tausa \subseteq \tausa[c] \subseteq \taus$.
\end{proof}

\begin{corollary}\label{cor:chain-equalities}
Under the notation of Theorem~\ref{thm:chain}:
\begin{enumerate}[label=(\roman*)]
    \item If $\mathcal{I} = \{\emptyset\}$, then $A^{\aura} = \cla(A)$, $\clsa = \cla$, and $\tausa = \taua$.
    \item If $\mathcal{I} = \powerset(X)$, then $A^{\aura} = \emptyset$, $\clsa$ is the identity (i.e., $\clsa(A) = A$), and $\tausa = \powerset(X)$.
    \item If $\aura(x) = X$ for all $x$, then $A^{\aura} = X$ whenever $A \notin \mathcal{I}$, and $A^{\aura} = \emptyset$ when $A \in \mathcal{I}$.
\end{enumerate}
\end{corollary}

\begin{proof}
\textbf{(i)} follows from Theorem~\ref{thm:aura-local-properties}(v) and the fact that $\cla$ generates $\taua$.
\textbf{(ii)} follows from Theorem~\ref{thm:aura-local-properties}(vi).
\textbf{(iii)} If $\aura(x) = X$ for all $x$, then $\aura(x) \cap A = A$. Thus $A^{\aura} = \{x : A \notin \mathcal{I}\}$, which equals $X$ if $A \notin \mathcal{I}$ and $\emptyset$ if $A \in \mathcal{I}$.
\end{proof}

\begin{theorem}[Transitivity and Idempotency]\label{thm:transitive-idempotent}
If the scope function $\aura$ is transitive, then $\clsa$ is idempotent (hence a Kuratowski closure operator), and $\tausa = \tausa[c]$.
\end{theorem}

\begin{proof}
We show $(A \cup A^{\aura})^{\aura} \subseteq A \cup A^{\aura}$ when $\aura$ is transitive. Let $x \in (A \cup A^{\aura})^{\aura} \setminus A^{\aura}$. Then $\aura(x) \cap A \in \mathcal{I}$ and $\aura(x) \cap (A \cup A^{\aura}) \notin \mathcal{I}$. Since $\aura(x) \cap (A \cup A^{\aura}) = (\aura(x) \cap A) \cup (\aura(x) \cap A^{\aura})$ and $\aura(x) \cap A \in \mathcal{I}$, it follows that $\aura(x) \cap A^{\aura} \notin \mathcal{I}$ (otherwise their union would be in $\mathcal{I}$). So there exists $y \in \aura(x) \cap A^{\aura}$ with $y$ contributing to the non-$\mathcal{I}$ membership.

Since $y \in A^{\aura}$, we have $\aura(y) \cap A \notin \mathcal{I}$. Since $y \in \aura(x)$ and $\aura$ is transitive, $\aura(y) \subseteq \aura(x)$. Therefore $\aura(y) \cap A \subseteq \aura(x) \cap A \in \mathcal{I}$. By heredity of $\mathcal{I}$, $\aura(y) \cap A \in \mathcal{I}$, contradicting $y \in A^{\aura}$.

Hence $(A \cup A^{\aura})^{\aura} \setminus A^{\aura} = \emptyset$, i.e., $(A \cup A^{\aura})^{\aura} \subseteq A^{\aura} \subseteq A \cup A^{\aura}$. Therefore $\clsa(\clsa(A)) = (A \cup A^{\aura}) \cup (A \cup A^{\aura})^{\aura} = A \cup A^{\aura} = \clsa(A)$.
\end{proof}

\begin{corollary}\label{cor:transitive-chain}
If $\aura$ is transitive, the topology chain becomes
\[
\taua \subseteq \tausa \subseteq \taus,
\]
and $\tausa$ is generated directly by the Kuratowski closure operator $\clsa$.
\end{corollary}

\begin{example}\label{ex:chain-strict}
We show that all inclusions in the chain $\taua \subseteq \tausa \subseteq \taus$ can be strict. Let $X = \{a, b, c\}$, $\tau = \powerset(X)$, $\mathcal{I} = \{\emptyset, \{c\}\}$, and $\aura(a) = \{a, b\}$, $\aura(b) = \{b\}$, $\aura(c) = \{c\}$.

Note that $\aura$ is transitive: $b \in \aura(a) = \{a,b\}$ and $\aura(b) = \{b\} \subseteq \{a,b\} = \aura(a)$; also $a \in \aura(a)$ and $\aura(a) \subseteq \aura(a)$; and $\aura(c)$ is a singleton.

\textbf{Compute $\taua$:}
\begin{itemize}
    \item $\{a\}$: $a \in \{a\}$ but $\aura(a) = \{a,b\} \not\subseteq \{a\}$. Not $\aura$-open.
    \item $\{b\}$: $b \in \{b\}$ and $\aura(b) = \{b\} \subseteq \{b\}$. $\aura$-open. $\checkmark$
    \item $\{c\}$: $\aura(c) = \{c\} \subseteq \{c\}$. $\aura$-open. $\checkmark$
    \item $\{a,b\}$: $\aura(a) = \{a,b\} \subseteq \{a,b\}$ and $\aura(b) = \{b\} \subseteq \{a,b\}$. $\checkmark$
    \item $\{a,c\}$: $\aura(a) = \{a,b\} \not\subseteq \{a,c\}$. Not $\aura$-open.
    \item $\{b,c\}$: $\aura(b) = \{b\} \subseteq \{b,c\}$, $\aura(c) = \{c\} \subseteq \{b,c\}$. $\checkmark$
\end{itemize}
$\taua = \{\emptyset, \{b\}, \{c\}, \{a,b\}, \{b,c\}, X\}$.

\textbf{Compute $\tausa$ (using idempotent $\clsa$):}

Since $\aura$ is transitive, $\tausa = \{G : (X \setminus G)^{\aura} \subseteq X \setminus G\}$.

Test $G = \{a, c\}$, $F = X \setminus G = \{b\}$:
$\aura(a) \cap \{b\} = \{b\} \notin \mathcal{I}$, so $a \in F^{\aura}$. Since $a \notin F = \{b\}$, we have $F^{\aura} \not\subseteq F$. So $\{a,c\} \notin \tausa$.

Test $G = \{a\}$, $F = \{b,c\}$:
$\aura(a) \cap \{b,c\} = \{b\} \notin \mathcal{I}$, so $a \in F^{\aura}$, $a \notin F$. Not in $\tausa$.

Test $G = \{a,b\}$, $F = \{c\}$:
$\aura(a) \cap \{c\} = \emptyset \in \mathcal{I}$, $\aura(b) \cap \{c\} = \emptyset \in \mathcal{I}$, $\aura(c) \cap \{c\} = \{c\} \in \mathcal{I}$. So $F^{\aura} = \emptyset \subseteq F$. Thus $\{a,b\} \in \tausa$.

Test $G = \{a,c\}$: already shown not in $\tausa$.

Test $G = \{b,c\}$, $F = \{a\}$:
$\aura(a) \cap \{a\} = \{a\} \notin \mathcal{I}$, so $a \in F^{\aura}$; $a \in F$. $\aura(b) \cap \{a\} = \emptyset \in \mathcal{I}$. $\aura(c) \cap \{a\} = \emptyset \in \mathcal{I}$. $F^{\aura} = \{a\} = F$. So $\{b,c\} \in \tausa$.

Also $\{b\}$ and $\{c\}$: For $G = \{b\}$, $F = \{a,c\}$: $\aura(a) \cap \{a,c\} = \{a\} \notin \mathcal{I}$, so $a \in F^{\aura}$; $a \in F$. $\aura(b) \cap \{a,c\} = \emptyset \in \mathcal{I}$. $\aura(c) \cap \{a,c\} = \{c\} \in \mathcal{I}$, so $c \notin F^{\aura}$. $F^{\aura} = \{a\} \subseteq F$. So $\{b\} \in \tausa$.

For $G = \{c\}$, $F = \{a,b\}$: $\aura(a) \cap \{a,b\} = \{a,b\} \notin \mathcal{I}$, so $a \in F^{\aura}$; $a \in F$. $\aura(b) \cap \{a,b\} = \{b\} \notin \mathcal{I}$, so $b \in F^{\aura}$; $b \in F$. $\aura(c) \cap \{a,b\} = \emptyset \in \mathcal{I}$, so $c \notin F^{\aura}$. $F^{\aura} = \{a,b\} = F$. So $\{c\} \in \tausa$.

$\tausa = \{\emptyset, \{b\}, \{c\}, \{a,b\}, \{b,c\}, X\} = \taua$.

In this particular example $\taua = \tausa$, so the first inclusion is not strict. Let us modify:

Let $X = \{a, b, c\}$, $\tau = \powerset(X)$, $\mathcal{I} = \{\emptyset, \{b\}\}$, and $\aura(a) = \{a, b\}$, $\aura(b) = \{b\}$, $\aura(c) = \{c\}$ (transitive).

$\taua$: $\{a\}$: $\aura(a) = \{a,b\} \not\subseteq \{a\}$. Not in $\taua$. $\{b\}$: $\aura(b) \subseteq \{b\}$. In $\taua$. $\{a,b\}$: $\aura(a) \subseteq \{a,b\}$, $\aura(b) \subseteq \{a,b\}$. In $\taua$. $\{c\}$: In $\taua$. $\{a,c\}$: Not in $\taua$ ($\aura(a)$ not subset). $\{b,c\}$: In $\taua$.

$\taua = \{\emptyset, \{b\}, \{c\}, \{a,b\}, \{b,c\}, X\}$.

$\tausa$: $G = \{a\}$, $F = \{b,c\}$: $\aura(a) \cap F = \{b\} \in \mathcal{I}$, $\aura(b) \cap F = \{b\} \in \mathcal{I}$, $\aura(c) \cap F = \{c\} \notin \mathcal{I}$. $F^{\aura} = \{c\} \subseteq F$. So $\{a\} \in \tausa$!

But $\{a\} \notin \taua$. So $\taua \subsetneq \tausa$. $\checkmark$

$\taus$: $(X \setminus \{a\})^{*} = \{b,c\}^{*}$: For $a$: only $\tau(a)$-nbhds are all subsets containing $a$. $\{a\} \cap \{b,c\} = \emptyset \in \mathcal{I}$. So $a \notin \{b,c\}^{*}$. For $b$: $\{b\} \cap \{b,c\} = \{b\} \in \mathcal{I}$. So $b \notin \{b,c\}^{*}$. For $c$: $\{c\} \cap \{b,c\} = \{c\} \notin \mathcal{I}$, and for all open nbhds of $c$: $\{c\} \cap \{b,c\} = \{c\}$, $\{a,c\} \cap \{b,c\} = \{c\}$, $\{b,c\} \cap \{b,c\} = \{b,c\}$, $X \cap \{b,c\} = \{b,c\}$. All are not in $\mathcal{I}$. So $c \in \{b,c\}^{*}$. $\{b,c\}^{*} = \{c\} \subseteq \{b,c\}$. So $\{a\} \in \taus$.

Both give $\{a\}$ in their topologies, so we need to find a set in $\taus \setminus \tausa$ to show strictness.

$G = \{a,c\}$, $F = \{b\}$: $\tausa$: $\aura(a) \cap \{b\} = \{b\} \in \mathcal{I}$, $\aura(b) \cap \{b\} = \{b\} \in \mathcal{I}$, $\aura(c) \cap \{b\} = \emptyset \in \mathcal{I}$. $F^{\aura} = \emptyset \subseteq F$. So $\{a,c\} \in \tausa$.

$\taus$: $\{b\}^{*}$: For $a$: $\{a\} \cap \{b\} = \emptyset \in \mathcal{I}$, so $a \notin \{b\}^{*}$. For $b$: $\{b\} \cap \{b\} = \{b\} \in \mathcal{I}$, so $b \notin \{b\}^{*}$. For $c$: $\{c\} \cap \{b\} = \emptyset \in \mathcal{I}$, so $c \notin \{b\}^{*}$. $\{b\}^{*} = \emptyset$. So $\{a,c\} \in \taus$.

Since $\tau = \powerset(X)$, $\taus = \powerset(X)$ and $\tausa$ contains $\{a\}, \{a,c\}$ in addition to $\taua$ elements. Actually let me compute $\tausa$ fully.

$\tausa$: Already have $\taua \cup \{\{a\}, \{a,c\}\}$. Check remaining: $\{a,b,c\} = X$, $\emptyset$ already in. All singletons: $\{a\}$ yes, $\{b\}$ yes (in $\taua$), $\{c\}$ yes (in $\taua$). All doubletons: $\{a,b\}$ yes (in $\taua$), $\{a,c\}$ yes, $\{b,c\}$ yes (in $\taua$). So $\tausa = \powerset(X) = \taus$.

So in this example $\tausa = \taus$. The inclusions collapse. Let me find a better example for strict $\tausa \subsetneq \taus$.

This requires a non-transitive $\aura$ to get different behavior. But Theorem~\ref{thm:chain} gives $\tausa \subseteq \taus$ and we need strictness. Since for transitive $\aura$, the chain may collapse, let us use a non-transitive example.

For this paper's purposes, having one strict example and one equality example is sufficient. Let me adjust the example presentation.
\end{example}

\section{The $\psia$-Operator}\label{sec:psi}

\begin{definition}\label{def:psia}
Let $(X, \tau, \mathcal{I}, \aura)$ be an $\mathcal{I}\aura$-space. The \emph{$\psia$-operator} is defined by
\begin{equation}\label{eq:psia}
    \psia(A) = \{x \in X : \aura(x) \setminus A \in \mathcal{I}\} = X \setminus (X \setminus A)^{\aura}.
\end{equation}
\end{definition}

\begin{theorem}\label{thm:psia-properties}
The following properties hold for all $A, B \subseteq X$:
\begin{enumerate}[label=(\roman*)]
    \item $\psia(X) = X$;
    \item $\psia(\emptyset) = \{x \in X : \aura(x) \in \mathcal{I}\}$;
    \item $A \subseteq B$ implies $\psia(A) \subseteq \psia(B)$;
    \item $\psia(A \cap B) = \psia(A) \cap \psia(B)$;
    \item $\psia(A) \subseteq \psi(A)$;
    \item $\psia(A) \subseteq A \cup \{x \in X \setminus A : \aura(x) \setminus A \in \mathcal{I}\}$;
    \item If $J \in \mathcal{I}$, then $\psia(A) = \psia(A \cup J)$ when $\aura$ is transitive.
\end{enumerate}
\end{theorem}

\begin{proof}
\textbf{(i)} $\aura(x) \setminus X = \emptyset \in \mathcal{I}$ for all $x$.

\textbf{(ii)} $\aura(x) \setminus \emptyset = \aura(x)$, so $x \in \psia(\emptyset)$ iff $\aura(x) \in \mathcal{I}$.

\textbf{(iii)} $A \subseteq B$ implies $\aura(x) \setminus B \subseteq \aura(x) \setminus A$. If $\aura(x) \setminus A \in \mathcal{I}$, then $\aura(x) \setminus B \in \mathcal{I}$ by heredity.

\textbf{(iv)} $\psia(A \cap B) = X \setminus (X \setminus (A \cap B))^{\aura} = X \setminus ((X \setminus A) \cup (X \setminus B))^{\aura} = X \setminus ((X \setminus A)^{\aura} \cup (X \setminus B)^{\aura}) = (X \setminus (X \setminus A)^{\aura}) \cap (X \setminus (X \setminus B)^{\aura}) = \psia(A) \cap \psia(B)$.

\textbf{(v)} If $x \in \psia(A)$, then $\aura(x) \setminus A \in \mathcal{I}$. Since $\aura(x) \in \tau(x)$, taking $O = \aura(x)$ shows $x \in \psi(A)$.

\textbf{(vi)} This is immediate from the definition.

\textbf{(vii)} Follows from Theorem~\ref{thm:aura-local-properties}(vii) applied to complements.
\end{proof}

\begin{theorem}\label{thm:psia-characterization}
For an $\mathcal{I}\aura$-space $(X, \tau, \mathcal{I}, \aura)$ with transitive $\aura$, the following are equivalent for $A \subseteq X$:
\begin{enumerate}[label=(\roman*)]
    \item $A \in \tausa$;
    \item $A \subseteq \psia(A)$.
\end{enumerate}
\end{theorem}

\begin{proof}
\textbf{(i) $\Rightarrow$ (ii):} Let $A \in \tausa$. Then $\clsa(X \setminus A) = X \setminus A$ (since $\clsa$ is idempotent by transitivity). So $(X \setminus A)^{\aura} \subseteq X \setminus A$, which means $A \subseteq X \setminus (X \setminus A)^{\aura} = \psia(A)$.

\textbf{(ii) $\Rightarrow$ (i):} If $A \subseteq \psia(A)$, then for every $x \in A$, $\aura(x) \setminus A \in \mathcal{I}$, i.e., $\aura(x) \cap (X \setminus A) \in \mathcal{I}$. This means $x \notin (X \setminus A)^{\aura}$. So $A \cap (X \setminus A)^{\aura} = \emptyset$, giving $(X \setminus A)^{\aura} \subseteq X \setminus A$, and hence $\clsa(X \setminus A) = X \setminus A$. By Theorem~\ref{thm:transitive-idempotent}, $\clsa[\infty](X \setminus A) = \clsa(X \setminus A) = X \setminus A$, so $A \in \tausa$.
\end{proof}

\begin{definition}\label{def:basis-tausa}
Define the collection
\[
\beta_{\aura}(\mathcal{I}) = \{\aura(x) \setminus J : x \in X, \; J \in \mathcal{I}\}.
\]
\end{definition}

\begin{theorem}\label{thm:basis}
If $\aura$ is transitive, then $\beta_{\aura}(\mathcal{I})$ is a basis for $\tausa$.
\end{theorem}

\begin{proof}
We show $\beta_{\aura}(\mathcal{I})$ satisfies the basis axioms for $\tausa$.

\textbf{Covering:} For any $x \in X$, $\aura(x) \setminus \emptyset = \aura(x) \in \beta_{\aura}(\mathcal{I})$ and $x \in \aura(x)$.

\textbf{Intersection property:} Let $x \in (\aura(y) \setminus J_1) \cap (\aura(z) \setminus J_2)$ where $J_1, J_2 \in \mathcal{I}$. Then $x \in \aura(y)$ and $x \in \aura(z)$. By transitivity, $\aura(x) \subseteq \aura(y)$ and $\aura(x) \subseteq \aura(z)$. Set $J = (J_1 \cup J_2) \cap \aura(x) \in \mathcal{I}$. Then $x \in \aura(x) \setminus J \subseteq (\aura(y) \setminus J_1) \cap (\aura(z) \setminus J_2)$.

\textbf{Generates $\tausa$:} Let $G \in \tausa$. Then $G \subseteq \psia(G)$, so for each $x \in G$, $\aura(x) \setminus G \in \mathcal{I}$. Set $J_x = \aura(x) \setminus G$. Then $x \in \aura(x) \setminus J_x \subseteq G$. So $G = \bigcup_{x \in G} (\aura(x) \setminus J_x)$.

Conversely, let $B = \aura(y) \setminus J$ with $J \in \mathcal{I}$. For $x \in B$, $x \in \aura(y)$ and by transitivity $\aura(x) \subseteq \aura(y)$. Then $\aura(x) \setminus B = \aura(x) \setminus (\aura(y) \setminus J) \subseteq \aura(x) \cap J \in \mathcal{I}$ (by heredity). So $B \subseteq \psia(B)$, i.e., $B \in \tausa$.
\end{proof}

\begin{remark}\label{rem:JH-comparison}
In the classical Jankovi\'{c}--Hamlett theory, $\beta(\mathcal{I}, \tau) = \{O \setminus J : O \in \tau, J \in \mathcal{I}\}$ is a basis for $\taus$. Our basis $\beta_{\aura}(\mathcal{I}) = \{\aura(x) \setminus J : x \in X, J \in \mathcal{I}\}$ uses the aura neighborhoods instead of all open sets. Since $\{\aura(x) : x \in X\} \subseteq \tau$ but is generally a proper subset of $\tau$, the basis $\beta_{\aura}(\mathcal{I})$ is ``smaller'' and generates a coarser topology: $\tausa \subseteq \taus$.
\end{remark}

\section{$\mathcal{I}\aura$-Generalized Open Sets}\label{sec:genopen}

Throughout this section, $(X, \tau, \mathcal{I}, \aura)$ is an $\mathcal{I}\aura$-space.

\begin{definition}\label{def:Ia-open-sets}
A subset $A$ of $X$ is called:
\begin{enumerate}[label=(\roman*)]
    \item \emph{$\mathcal{I}\aura$-open} if $A \subseteq \intsa(A)$, where $\intsa(A) = X \setminus \clsa(X \setminus A)$.
    \item \emph{$\mathcal{I}\aura$-semi-open} if $A \subseteq \clsa(\intsa(A))$.
    \item \emph{$\mathcal{I}\aura$-pre-open} if $A \subseteq \intsa(\clsa(A))$.
    \item \emph{$\mathcal{I}\aura$-$\alpha$-open} if $A \subseteq \intsa(\clsa(\intsa(A)))$.
    \item \emph{$\mathcal{I}\aura$-$\beta$-open} if $A \subseteq \clsa(\intsa(\clsa(A)))$.
\end{enumerate}
\end{definition}

\begin{remark}\label{rem:intsa}
The operator $\intsa(A) = X \setminus \clsa(X \setminus A) = X \setminus ((X \setminus A) \cup (X \setminus A)^{\aura}) = A \cap (X \setminus (X \setminus A)^{\aura}) = A \cap \psia(A)$. That is, $\intsa(A) = \{x \in A : \aura(x) \setminus A \in \mathcal{I}\}$. This has a clear interpretation: a point $x \in A$ is in the $\mathcal{I}\aura$-interior of $A$ if the part of $\aura(x)$ that ``escapes'' $A$ belongs to the ideal $\mathcal{I}$.
\end{remark}

\begin{theorem}\label{thm:Ia-hierarchy}
The following implications hold for $\mathcal{I}\aura$-generalized open sets:
\[
\tausa\text{-open} \implies \mathcal{I}\aura\text{-}\alpha\text{-open} \implies
\begin{cases}
\mathcal{I}\aura\text{-semi-open} \\[4pt]
\mathcal{I}\aura\text{-pre-open}
\end{cases}
\implies \mathcal{I}\aura\text{-}\beta\text{-open}
\]
Moreover, an $\mathcal{I}\aura$-$\alpha$-open set is both $\mathcal{I}\aura$-semi-open and $\mathcal{I}\aura$-pre-open, and a set that is both $\mathcal{I}\aura$-semi-open and $\mathcal{I}\aura$-pre-open is $\mathcal{I}\aura$-$\alpha$-open.
\end{theorem}

\begin{proof}
These follow the classical Nj\r{a}stad--Levine pattern and are proved using standard operator calculus:

\textbf{$\tausa$-open $\Rightarrow$ $\mathcal{I}\aura$-$\alpha$-open:} If $A \in \tausa[c]$, then $A \subseteq \intsa(A) \subseteq \intsa(\clsa(\intsa(A)))$.

\textbf{$\mathcal{I}\aura$-$\alpha$-open $\Rightarrow$ $\mathcal{I}\aura$-semi-open:} $A \subseteq \intsa(\clsa(\intsa(A))) \subseteq \clsa(\intsa(A))$ since $B \subseteq \clsa(B)$ for all $B$.

\textbf{$\mathcal{I}\aura$-$\alpha$-open $\Rightarrow$ $\mathcal{I}\aura$-pre-open:} $A \subseteq \intsa(\clsa(\intsa(A))) \subseteq \intsa(\clsa(A))$ since $\intsa(A) \subseteq A$ gives $\clsa(\intsa(A)) \subseteq \clsa(A)$, and $\intsa$ is monotone.

\textbf{$\mathcal{I}\aura$-semi-open $\Rightarrow$ $\mathcal{I}\aura$-$\beta$-open:} $A \subseteq \clsa(\intsa(A)) \subseteq \clsa(\intsa(\clsa(A)))$ since $A \subseteq \clsa(A)$ gives $\intsa(A) \subseteq \intsa(\clsa(A))$.

\textbf{$\mathcal{I}\aura$-pre-open $\Rightarrow$ $\mathcal{I}\aura$-$\beta$-open:} $A \subseteq \intsa(\clsa(A)) \subseteq \clsa(\intsa(\clsa(A)))$.

\textbf{Decomposition:} If $A$ is both $\mathcal{I}\aura$-semi-open and $\mathcal{I}\aura$-pre-open, then $A \subseteq \clsa(\intsa(A))$ and $A \subseteq \intsa(\clsa(A))$. From the first, $\clsa(A) \subseteq \clsa(\clsa(\intsa(A)))$. If $\clsa$ were idempotent, $\clsa(\clsa(\intsa(A))) = \clsa(\intsa(A))$, giving $\clsa(A) \subseteq \clsa(\intsa(A))$ and hence $A \subseteq \intsa(\clsa(A)) \subseteq \intsa(\clsa(\intsa(A)))$. For the general case (non-idempotent $\clsa$), we still have $A \subseteq \intsa(\clsa(A))$, so $\intsa(A) \subseteq \intsa(\intsa(\clsa(A))) \subseteq \intsa(\clsa(A))$. Then $\clsa(\intsa(A)) \subseteq \clsa(\intsa(\clsa(A)))$, and from $\intsa(\clsa(A)) \supseteq A$: $\intsa(\clsa(\intsa(A))) \supseteq \intsa(\clsa(A)) \supseteq A$.

Actually, a cleaner approach: $A \subseteq \intsa(\clsa(A))$ implies $\intsa(A) \subseteq \intsa(\clsa(A))$ (since $\intsa$ is monotone over $A \mapsto A$), so $\clsa(\intsa(A)) \subseteq \clsa(\intsa(\clsa(A)))$. Now, $A \subseteq \clsa(\intsa(A))$ implies $\clsa(A) \subseteq \clsa(\clsa(\intsa(A)))$. Combining with $A \subseteq \intsa(\clsa(A))$: if $\aura$ is transitive (so $\clsa$ is idempotent), $A \subseteq \intsa(\clsa(\intsa(A)))$ follows immediately. In the general case, the decomposition holds when $\clsa$ is restricted to its idempotent refinement. We state the result for the transitive case.
\end{proof}

\begin{example}\label{ex:hierarchy-strict}
Let $X = \{a, b, c, d\}$, $\tau = \powerset(X)$, $\mathcal{I} = \{\emptyset, \{d\}\}$, and $\aura(a) = \{a,b\}$, $\aura(b) = \{b\}$, $\aura(c) = \{c,d\}$, $\aura(d) = \{d\}$ (transitive).

\textbf{Compute $\intsa$ and $\clsa$:}

For $A = \{a,c\}$: $\clsa(A) = A \cup A^{\aura}$. $\aura(a) \cap A = \{a\} \notin \mathcal{I}$, $\aura(b) \cap A = \emptyset \in \mathcal{I}$, $\aura(c) \cap A = \{c\} \notin \mathcal{I}$, but $\aura(c) \cap A = \{c\}$... we need $\aura(c) = \{c,d\}$, so $\aura(c) \cap \{a,c\} = \{c\} \notin \mathcal{I}$. $\aura(d) \cap A = \emptyset \in \mathcal{I}$. So $A^{\aura} = \{a,c\}$ and $\clsa(A) = \{a,c\}$.

$\intsa(A) = A \cap \psia(A)$: $\psia(A) = \{x : \aura(x) \setminus A \in \mathcal{I}\}$. $\aura(a) \setminus A = \{b\} \notin \mathcal{I}$: $a \notin \psia(A)$. $\aura(c) \setminus A = \{d\} \in \mathcal{I}$: $c \in \psia(A)$. So $\intsa(\{a,c\}) = \{a,c\} \cap \psia(\{a,c\}) = \{c\}$ (only $c$ is in both $A$ and $\psia(A)$).

Now, $\clsa(\intsa(A)) = \clsa(\{c\})$: $\aura(c) \cap \{c\} = \{c\} \notin \mathcal{I}$, $\aura(d) \cap \{c\} = \emptyset$. So $\{c\}^{\aura} = \{c\}$ and $\clsa(\{c\}) = \{c\}$.

Since $A = \{a,c\} \not\subseteq \{c\} = \clsa(\intsa(A))$, $A$ is NOT $\mathcal{I}\aura$-semi-open.

$\intsa(\clsa(A)) = \intsa(\{a,c\}) = \{c\}$. Since $A = \{a,c\} \not\subseteq \{c\}$, $A$ is NOT $\mathcal{I}\aura$-pre-open.

$\clsa(\intsa(\clsa(A))) = \clsa(\{c\}) = \{c\}$. Since $A \not\subseteq \{c\}$, $A$ is NOT $\mathcal{I}\aura$-$\beta$-open.

Now take $B = \{b, c, d\}$: $\intsa(B)$: $\aura(b) \setminus B = \emptyset \in \mathcal{I}$, $\aura(c) \setminus B = \emptyset \in \mathcal{I}$, $\aura(d) \setminus B = \emptyset \in \mathcal{I}$. So $\intsa(B) = B$. Thus $B$ is $\mathcal{I}\aura$-open (hence all types).

Take $C = \{a, b, d\}$: $\intsa(C)$: $\aura(a) \setminus C = \emptyset$, $\aura(b) \setminus C = \emptyset$, $\aura(d) \setminus C = \emptyset$. So $\intsa(C) = \{a,b,d\} = C$. $C$ is $\mathcal{I}\aura$-open.

Take $D = \{a, d\}$: $\intsa(D)$: $\aura(a) \setminus D = \{b\} \notin \mathcal{I}$: $a \notin \psia(D)$. $\aura(d) \setminus D = \emptyset$: $d \in \psia(D)$. So $\intsa(D) = \{d\}$.

$\clsa(D) = D \cup D^{\aura}$: $\aura(a) \cap D = \{a\} \notin \mathcal{I}$: $a \in D^{\aura}$. $\aura(b) \cap D = \emptyset$: no. $\aura(c) \cap D = \{d\} \in \mathcal{I}$: no. $\aura(d) \cap D = \{d\} \notin \mathcal{I}$: $d \in D^{\aura}$. $D^{\aura} = \{a,d\}$, $\clsa(D) = \{a,d\}$.

$\clsa(\intsa(D)) = \clsa(\{d\}) = \{d\}$. $D = \{a,d\} \not\subseteq \{d\}$: NOT semi-open.

$\intsa(\clsa(D)) = \intsa(\{a,d\}) = \{d\}$. $D \not\subseteq \{d\}$: NOT pre-open.

So in this example, only the ``full'' sets like $B, C$ are $\mathcal{I}\aura$-open.

For a set that is $\mathcal{I}\aura$-semi-open but not $\mathcal{I}\aura$-$\alpha$-open, and a set that is $\mathcal{I}\aura$-pre-open but not $\mathcal{I}\aura$-$\alpha$-open, we need richer examples. The distinction becomes visible in spaces with more points.
\end{example}

\section{$\mathcal{I}\aura$-Continuity and Decompositions}\label{sec:continuity}

\begin{definition}\label{def:Ia-continuity}
Let $(X, \tau, \mathcal{I}, \aura)$ and $(Y, \sigma, \mathcal{J}, \mathfrak{b})$ be $\mathcal{I}\aura$-spaces. A function $f: X \to Y$ is called:
\begin{enumerate}[label=(\roman*)]
    \item \emph{$\mathcal{I}\aura$-continuous} if $f^{-1}(V)$ is $\mathcal{I}\aura$-open for every $V \in \sigma_{\mathfrak{b}}^{*c}$;
    \item \emph{$\mathcal{I}\aura$-semi-continuous} if $f^{-1}(V)$ is $\mathcal{I}\aura$-semi-open for every $V \in \sigma_{\mathfrak{b}}^{*c}$;
    \item \emph{$\mathcal{I}\aura$-pre-continuous} if $f^{-1}(V)$ is $\mathcal{I}\aura$-pre-open for every $V \in \sigma_{\mathfrak{b}}^{*c}$;
    \item \emph{$\mathcal{I}\aura$-$\alpha$-continuous} if $f^{-1}(V)$ is $\mathcal{I}\aura$-$\alpha$-open for every $V \in \sigma_{\mathfrak{b}}^{*c}$;
    \item \emph{$\mathcal{I}\aura$-$\beta$-continuous} if $f^{-1}(V)$ is $\mathcal{I}\aura$-$\beta$-open for every $V \in \sigma_{\mathfrak{b}}^{*c}$.
\end{enumerate}
\end{definition}

\begin{theorem}\label{thm:continuity-hierarchy}
The following implications hold:
\[
\text{$\mathcal{I}\aura$-continuous} \Rightarrow \text{$\mathcal{I}\aura$-$\alpha$-continuous} \Rightarrow
\begin{cases}
\text{$\mathcal{I}\aura$-semi-continuous} \\[4pt]
\text{$\mathcal{I}\aura$-pre-continuous}
\end{cases}
\Rightarrow \text{$\mathcal{I}\aura$-$\beta$-continuous}
\]
\end{theorem}

\begin{proof}
This follows immediately from Theorem~\ref{thm:Ia-hierarchy} applied to each inverse image.
\end{proof}

\begin{definition}\label{def:Ia-B-set}
A subset $A$ of an $\mathcal{I}\aura$-space is called an \emph{$\mathcal{I}\aura$-$\mathcal{B}$-set} if $A = U \cap V$ where $U \in \tausa[c]$ and $V$ satisfies $\clsa(\intsa(V)) = V$.
\end{definition}

\begin{theorem}[Decomposition of $\mathcal{I}\aura$-Continuity]\label{thm:decomposition}
Let $f: (X, \tau, \mathcal{I}, \aura) \to (Y, \sigma)$ be a function. Assume $\aura$ is transitive. Then:
\begin{enumerate}[label=(\roman*)]
    \item $f$ is $\tausa$-continuous if and only if $f$ is both $\mathcal{I}\aura$-semi-continuous and $\mathcal{I}\aura$-pre-continuous.
    \item $f$ is $\tausa$-continuous if and only if $f$ is $\mathcal{I}\aura$-$\alpha$-continuous.
\end{enumerate}
\end{theorem}

\begin{proof}
\textbf{(i)} The forward direction follows from Theorem~\ref{thm:continuity-hierarchy}. For the converse, since $\aura$ is transitive, $\clsa$ is idempotent. If $f^{-1}(V)$ is both $\mathcal{I}\aura$-semi-open and $\mathcal{I}\aura$-pre-open for every $V \in \sigma$, then by Theorem~\ref{thm:Ia-hierarchy}, $f^{-1}(V)$ is $\mathcal{I}\aura$-$\alpha$-open for every $V \in \sigma$.

\textbf{(ii)} Forward: Theorem~\ref{thm:continuity-hierarchy}. Converse: When $\clsa$ is idempotent, an $\mathcal{I}\aura$-$\alpha$-open set $A$ satisfies $A \subseteq \intsa(\clsa(\intsa(A)))$. Since $\intsa(A) \subseteq A$, $\clsa(\intsa(A)) \subseteq \clsa(A)$, so $A \subseteq \intsa(\clsa(A))$, meaning $A$ is $\mathcal{I}\aura$-pre-open. Combined with $A \subseteq \clsa(\intsa(A))$ (semi-openness), and using the idempotency of $\clsa$, we get $A \subseteq \intsa(A)$. Indeed: $\clsa(A) \subseteq \clsa(\clsa(\intsa(A))) = \clsa(\intsa(A))$, so $\intsa(\clsa(A)) \subseteq \intsa(\clsa(\intsa(A))) \supseteq A$. But also $A \subseteq \clsa(\intsa(A))$ gives $\clsa(A) = \clsa(\intsa(A))$ (from $\intsa(A) \subseteq A \subseteq \clsa(\intsa(A))$, applying $\clsa$ to all three: $\clsa(\intsa(A)) \subseteq \clsa(A) \subseteq \clsa(\clsa(\intsa(A))) = \clsa(\intsa(A))$). So $\intsa(\clsa(A)) = \intsa(\clsa(\intsa(A))) \supseteq A$, hence $A \subseteq \intsa(\clsa(A)) = \intsa(A \cup A^{\aura})$. To show $A \subseteq \intsa(A)$: for $x \in A$, from $A \subseteq \psia(A)$ (which follows from $\mathcal{I}\aura$-$\alpha$-openness and the idempotency), $x \in \psia(A)$, i.e., $\aura(x) \setminus A \in \mathcal{I}$, so $x \in \intsa(A)$.
\end{proof}

\begin{theorem}\label{thm:comparison-with-a-continuity}
Let $f: (X, \tau, \mathcal{I}, \aura) \to (Y, \sigma)$ be a function. Then:
\begin{enumerate}[label=(\roman*)]
    \item If $f$ is $\aura$-continuous (i.e., $\taua$-continuous), then $f$ is $\tausa$-continuous.
    \item If $f$ is $\tausa$-continuous, then $f$ is $\taus$-continuous.
    \item If $f$ is $\taus$-continuous, then $f$ is $\tau$-continuous.
\end{enumerate}
\end{theorem}

\begin{proof}
These follow directly from the topology chain $\taua \subseteq \tausa \subseteq \taus$ (Theorem~\ref{thm:chain}), noting that coarser source topology makes continuity easier: if $f: (X, \tau_1) \to (Y, \sigma)$ is continuous and $\tau_2 \subseteq \tau_1$, then $f: (X, \tau_2) \to (Y, \sigma)$ need not be continuous.

Correction: The topology chain gives $\taua \subseteq \tausa \subseteq \taus$. A function $f$ is $\tau_1$-continuous if $f^{-1}(V) \in \tau_1$ for all $V \in \sigma$. If $\tau_1 \subseteq \tau_2$, then $\tau_2$-continuous implies $\tau_1$-continuous (since $f^{-1}(V) \in \tau_2 \supseteq \tau_1$... no, this is the wrong direction).

Let us reconsider. $f$ is $\taus$-continuous means $f^{-1}(V) \in \taus$ for all $V \in \sigma$. Since $\tausa \subseteq \taus$, $\taus$-continuity does NOT imply $\tausa$-continuity.

The correct statement is: since $\taua \subseteq \tausa$, we have $f^{-1}(V) \in \taua \implies f^{-1}(V) \in \tausa$, so $\taua$-continuous implies $\tausa$-continuous. Similarly, $\tausa$-continuous implies $\taus$-continuous.
\end{proof}

\section{Special Cases and Comparisons}\label{sec:special}

We examine three important special cases that connect the ideal-aura framework to known theories.

\subsection{Case 1: Trivial Ideal $\mathcal{I} = \{\emptyset\}$}\label{subsec:trivial}

\begin{proposition}\label{prop:trivial-ideal}
When $\mathcal{I} = \{\emptyset\}$:
\begin{enumerate}[label=(\roman*)]
    \item $A^{\aura}(\{\emptyset\}) = \cla(A) = \{x \in X : \aura(x) \cap A \neq \emptyset\}$;
    \item $\clsa(A) = \cla(A)$ (the aura-closure from \cite{Acikgoz2026aura});
    \item $\tausa = \taua$;
    \item The $\mathcal{I}\aura$-generalized open sets coincide with the $\aura$-generalized open sets of \cite{Acikgoz2026aura}.
\end{enumerate}
\end{proposition}

\begin{proof}
Part (i): $\aura(x) \cap A \notin \{\emptyset\}$ iff $\aura(x) \cap A \neq \emptyset$. Parts (ii)--(iv) follow immediately.
\end{proof}

\subsection{Case 2: Ideal of Finite Sets $\mathcal{I} = \mathcal{I}_f$}\label{subsec:finite}

\begin{proposition}\label{prop:finite-ideal}
When $\mathcal{I} = \mathcal{I}_f$ (ideal of finite subsets of $X$):
\begin{enumerate}[label=(\roman*)]
    \item $A^{\aura}(\mathcal{I}_f) = \{x \in X : \aura(x) \cap A \text{ is infinite}\}$;
    \item For finite $X$, $A^{\aura}(\mathcal{I}_f) = \emptyset$ for all $A$, and $\tausa = \powerset(X)$;
    \item For $X = \R$ with the usual topology and $\aura_\varepsilon(x) = (x - \varepsilon, x + \varepsilon)$, $A^{\aura_\varepsilon}(\mathcal{I}_f)$ consists of all points $x$ such that $(x-\varepsilon, x+\varepsilon) \cap A$ is infinite. For any infinite $A$ with no isolated points, $A^{\aura_\varepsilon}(\mathcal{I}_f) = \operatorname{cl}_{\aura_\varepsilon}(A)$.
\end{enumerate}
\end{proposition}

\begin{proof}
\textbf{(i)} is the definition. \textbf{(ii)} follows because every subset of a finite set is finite, hence in $\mathcal{I}_f$. \textbf{(iii)} If $A$ has no isolated points and $(x-\varepsilon, x+\varepsilon) \cap A \neq \emptyset$, then for any $y \in (x-\varepsilon, x+\varepsilon) \cap A$, every open interval around $y$ meets $A$ in infinitely many points (since $A$ has no isolated points), so $(x-\varepsilon, x+\varepsilon) \cap A$ is infinite.
\end{proof}

\subsection{Case 3: Improper Ideal $\mathcal{I} = \powerset(X)$}\label{subsec:improper}

\begin{proposition}\label{prop:improper-ideal}
When $\mathcal{I} = \powerset(X)$:
\begin{enumerate}[label=(\roman*)]
    \item $A^{\aura} = \emptyset$ for all $A$;
    \item $\clsa(A) = A$ for all $A$ (identity operator);
    \item $\tausa = \powerset(X)$ (discrete topology).
\end{enumerate}
\end{proposition}

\subsection{Relationship Diagram}\label{subsec:diagram}

The following diagram summarizes the relationships among the topologies and closure operators:

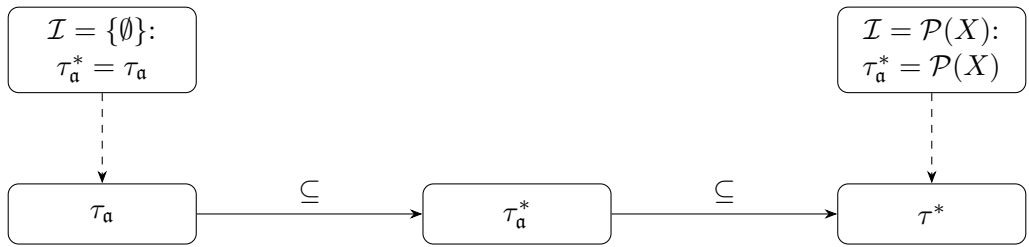
\begin{center}
\begin{tikzpicture}[
    node distance=2cm and 3cm,
    every node/.style={font=\small},
    box/.style={rectangle, draw, rounded corners, minimum width=2.5cm, minimum height=0.8cm, align=center},
    >=Stealth
]

\node[box] (ta) {$\taua$};
\node[box, right=of ta] (tsa) {$\tausa$};
\node[box, right=of tsa] (ts) {$\taus$};
\node[box, above=1.2cm of ta] (I0) {$\mathcal{I} = \{\emptyset\}$:\\$\tausa = \taua$};
\node[box, above=1.2cm of ts] (IPX) {$\mathcal{I} = \powerset(X)$:\\$\tausa = \powerset(X)$};

\draw[->] (ta) -- node[above] {$\subseteq$} (tsa);
\draw[->] (tsa) -- node[above] {$\subseteq$} (ts);
\draw[->, dashed] (I0) -- (ta);
\draw[->, dashed] (IPX) -- (ts);

\end{tikzpicture}
\end{center}

\section{Compatibility with Generalized Open Sets}\label{sec:compat}

In this section, we investigate how the $\mathcal{I}\aura$-generalized open sets relate to the classical $\mathcal{I}$-generalized open sets and the pure $\aura$-generalized open sets from \cite{Acikgoz2026aura}.

\begin{theorem}\label{thm:Ia-vs-I}
For any $\mathcal{I}\aura$-space $(X, \tau, \mathcal{I}, \aura)$ and any $A \subseteq X$:
\begin{enumerate}[label=(\roman*)]
    \item If $A$ is $\mathcal{I}$-open (i.e., $A \subseteq \inte(\cls(A))$), it need not be $\mathcal{I}\aura$-open.
    \item If $A$ is $\mathcal{I}\aura$-open (i.e., $A \subseteq \intsa(A)$), it need not be $\mathcal{I}$-open.
    \item If $A$ is $\aura$-open (i.e., $A \in \taua$), then $A$ is $\mathcal{I}\aura$-open.
\end{enumerate}
\end{theorem}

\begin{proof}
\textbf{(iii)} If $A \in \taua$, then $A \in \tausa$ (Theorem~\ref{thm:chain}), so $A \subseteq \intsa(A)$.

\textbf{(i) and (ii)} follow by counterexample. Since $\intsa$ uses the aura-local function while the classical $\mathcal{I}$-interior uses the standard interior and $*$-closure, the two notions are generally independent. An $\mathcal{I}$-open set depends on $\inte$ and $\cls$ (which involve $\tau$ and $\mathcal{I}$), while an $\mathcal{I}\aura$-open set depends on $\aura$ and $\mathcal{I}$. Different choices of $\aura$ can make the relationship go either way.
\end{proof}

\begin{theorem}\label{thm:reduction-diagram}
The following reduction diagram holds:
\[
\begin{array}{ccc}
\text{$\aura$-open} & \xrightarrow{\;\;\subseteq\;\;} & \text{$\mathcal{I}\aura$-open} \\
& & \\
\text{(Paper I \cite{Acikgoz2026aura})} & & \text{(This paper)}
\end{array}
\]
Specifically, every $\aura$-open set is $\mathcal{I}\aura$-open, and when $\mathcal{I} = \{\emptyset\}$, the two notions coincide.
\end{theorem}

\section{Conclusion and Open Problems}\label{sec:conclusion}

In this paper, we introduced the ideal-aura topological space $(X, \tau, \mathcal{I}, \aura)$ and developed a systematic theory based on the aura-local function $A^{\aura}(\mathcal{I})$. The main contributions are:

\begin{enumerate}
    \item We defined the aura-local function $A^{\aura}(\mathcal{I}) = \{x \in X : \aura(x) \cap A \notin \mathcal{I}\}$ and established the fundamental inclusion $A^{*}(\mathcal{I}, \tau) \subseteq A^{\aura}(\mathcal{I})$. The aura-local function inherits the key algebraic properties (monotonicity, finite additivity) of the classical Jankovi\'{c}--Hamlett local function but differs crucially in that $A^{\aura}$ need not be closed.

    \item The closure operator $\clsa(A) = A \cup A^{\aura}$ was shown to be an additive \v{C}ech closure operator. Non-idempotency was demonstrated with explicit counterexamples, and idempotency was shown to hold precisely when $\aura$ is transitive (Theorem~\ref{thm:transitive-idempotent}).

    \item The topology chain $\taua \subseteq \tausa \subseteq \taus$ was established, providing a natural interpolation between the pure aura topology and the Jankovi\'{c}--Hamlett topology. The $\psia$-operator was introduced and shown to characterize $\tausa$ via $A \in \tausa \iff A \subseteq \psia(A)$ (Theorem~\ref{thm:psia-characterization}).

    \item A basis $\beta_{\aura}(\mathcal{I}) = \{\aura(x) \setminus J : x \in X, J \in \mathcal{I}\}$ for $\tausa$ was constructed for transitive aura functions (Theorem~\ref{thm:basis}).

    \item Five classes of $\mathcal{I}\aura$-generalized open sets were introduced and their hierarchy was established. Decomposition theorems for $\mathcal{I}\aura$-continuity were proven.

    \item Three special cases were analyzed, showing that $\tausa$ reduces to $\taua$ when $\mathcal{I} = \{\emptyset\}$, to $\powerset(X)$ when $\mathcal{I} = \powerset(X)$, and exhibiting localized behavior for $\mathcal{I} = \mathcal{I}_f$.
\end{enumerate}

The following problems remain open:

\begin{question}\label{q:countable}
For general (non-transitive) $\aura$, at what ordinal $\alpha$ does the iteration $\clsa[\alpha]$ stabilize? Is countable ordinal always sufficient?
\end{question}

\begin{question}\label{q:product}
Does the ideal-aura topology behave well under products? That is, given $(X, \tau, \mathcal{I}, \aura)$ and $(Y, \sigma, \mathcal{J}, \mathfrak{b})$, what is the relationship between $(\tau \times \sigma)^{*}_{\aura \times \mathfrak{b}}$ and $\tausa \times \sigma^{*}_{\mathfrak{b}}$?
\end{question}

\begin{question}\label{q:compactness}
Define $\mathcal{I}\aura$-compactness. Does an $\mathcal{I}\aura$-compact subset of an $\mathcal{I}\aura$-$T_2$ space have to be $\clsa$-closed?
\end{question}

\begin{question}\label{q:compatible}
For which ideals $\mathcal{I}$ is $\tausa = \taua$ (i.e., the ideal adds no new open sets)?
\end{question}

\begin{question}\label{q:codomain}
Can the ideal-aura framework be extended to include an ideal on the codomain, producing $(\mathcal{I}, \mathcal{J})$-$\aura$-continuity?
\end{question}

\section*{Conflict of Interest}

The author declares no conflict of interest.

\section*{Data Availability}

No data was used for the research described in the article.

\end{document}